\documentclass[10pt,a4paper,english,reqno]{amsart}
\usepackage{amsmath}
\usepackage{a4}
\usepackage{amssymb}
\usepackage{amsthm}
\usepackage{bbm}
\usepackage{bm}
\usepackage{dsfont}
\usepackage[T1]{fontenc}

\usepackage[utf8]{inputenc}
\usepackage[scaled=0.86]{helvet}
\usepackage{parskip}
\usepackage[unicode=true,pdfusetitle,
 bookmarks=true,bookmarksnumbered=false,bookmarksopen=false,
 breaklinks=false,pdfborder={0 0 1}%,backref=section,colorlinks=false
 ]
 {hyperref}

\usepackage[english]{babel}
\usepackage{latexsym}
\usepackage{mathtools}
\usepackage{enumitem}
\usepackage{cases}

\newcommand{\N}{\mathbb N}

\makeatletter

\pdfpageheight\paperheight
\pdfpagewidth\paperwidth

\theoremstyle{plain}
\newtheorem{thm}{Theorem}[section] % first theorem in section 1 will be 1.1
\theoremstyle{definition}
\newtheorem{Def}[thm]{\protect\definitionname}

\theoremstyle{plain}

\theoremstyle{plain}

\newtheorem{lemma}[thm]{Lemma}

\theoremstyle{definition}

\theoremstyle{remark}
\newtheorem{Rem}[thm]{Remark}
\newtheorem{Ex}[thm]{Example}

\usepackage{a4wide}
\usepackage{bbm}
\usepackage[arrow, matrix, curve]{xy}
\usepackage{tikz}

\newcommand\upperleft[3]{\prescript{#1}{}{\mathrlap{\smash{#2#3}}\phantom{#3}}}

\renewcommand{\1}{\mathbbm{1}}

\usepackage{amssymb}
\usepackage{mathtools}

\makeatother

  \providecommand{\definitionname}{Definition}

\begin{document}
\title[Mean convergence for intermediately trimmed Birkhoff sums]
{Mean convergence for intermediately trimmed Birkhoff sums of observables with regularly varying tails}

\author[Kesseb\"ohmer]{Marc Kesseb\"ohmer}
  \address{Universit\"at Bremen, Fachbereich 3 -- Mathematik und Informatik, Bibliothekstr. 1, 28359 Bremen, Germany}
  \email{\href{mailto:mhk@math.uni-bremen.de}{mhk@math.uni-bremen.de}}
\author[Schindler]{Tanja Schindler}
\address{Australian National University, Research School of Finance, Actuarial Studies and Statistics, 26C Kingsley St,
Acton ACT 2601, Australia}
  \email{\href{mailto:tanja.schindler@anu.edu.au}{tanja.schindler@anu.edu.au}}

\date{\today}

\begin{abstract}
On a measure theoretical dynamical system
with  spectral gap property we consider  non-integrable observables 
with regularly varying tails and fulfilling a mild mixing condition. 
We show that the normed trimmed sum process of these observables then converges in  mean. This result is   new also for the special case of i.i.d.\ random variables
 and contrasts  the general case where mean convergence might fail
even though a strong law of large numbers holds.
To illuminate the required mixing condition we give an explicit example of a dynamical system fulfilling a spectral gap  property and an observable with regularly varying tails but without the assumed mixing condition  such that  mean convergence fails.
\end{abstract}
\maketitle

\section{Introduction and statement of main results}
We consider an ergodic measure preserving dynamical system $\left(\Omega,\mathcal{A},T, \mu\right)$ 
with $\mu$ a probability measure and stochastic processes given by the Birkhoff sums $\mathsf{S}_n\chi\coloneqq\sum_{k=1}^n\chi\circ  T^{k-1}$, $n\in \N$, with $\mathsf{S}_0\chi=0$ for some function $\chi\in \mathcal{M}(\mathcal{A})^{+}\coloneqq \{f:\Omega\to \mathbb{R}_{\geq 0}\colon f \mbox{ is }\mathcal{A}\mbox{-measurable}\}$ sometimes called  observable. 
If $\int\chi\;\mathrm{d}\mu$ is finite,  then we obtain by {\em Birkhoff's ergodic theorem} --\,combining pointwise and mean convergence\,-- that $\mu$-almost surely (a.s.)  
\begin{align*}
 \lim_{n\to\infty}\frac{\mathsf{S}_n\chi}{\int\mathsf{S}_n\chi\;\mathrm{d}\mu}=1,
\end{align*}
i.e.\  the strong law of large numbers is fulfilled with norming sequence $(\int\mathsf{S}_n\chi\;\mathrm{d}\mu)$, whereas 
in the case $\int\chi\;\mathrm{d}\mu=\infty$, 
Aaronson ruled out the possibility of a strong law of large numbers
no matter which norming sequence  
we choose,
see \cite{aaronson_ergodic_1977}.
However, 
in certain cases after deleting a number
of the largest summands from the partial $n$-sums a strong law of large numbers holds.
More precisely,
for each $n\in\mathbb{N}$ we choose a permutation $\sigma\in\mathcal{S}_{n}$
of $\left\{ 0,\ldots,n-1\right\} $ with $\chi\circ  T^{\sigma\left(1\right)}\geq \chi\circ  T^{\sigma\left(2\right)}\geq\ldots\geq \chi\circ  T^{\sigma\left(n\right)}$
and for given $b_{n}\in\mathbb{N}_{0}$ we define 
\begin{align*}
\mathsf{S}_{n}^{b_{n}}\chi & \coloneqq\sum_{k=b_n}^{n-1}\chi\circ  T^{\sigma\left(k\right)}.
\end{align*}

If $b_{n}=r\in\mathbb{N}$ is fixed for all $n\in\mathbb{N}$, then $\left(\mathsf{S}_{n}^{r}\varphi\right)$
is called a \emph{lightly trimmed sum} \emph{process}.
If we allow the sequence $\left(b_{n}\right)\in\mathbb{N}^{\mathbb{N}}$
to diverge to infinity such that $b_{n}=o\left(n\right)$, i.e.\ $\lim_{n\rightarrow\infty}b_{n}/n=0$,
then $\left(\mathsf{S}_{n}^{b_n}\varphi\right)$ is called
an \emph{intermediately} (also \emph{moderately}) \emph{trimmed sum process}.

The special case of \emph{regularly varying} tail variables with index strictly between $-1$ and $0$
has been considered by the authors in \cite{kessebohmer_strong_2018}.
That is for  $F:x\mapsto   \mu\left(\chi\leq x\right)$ denoting the distribution function of $\chi$ we  require that $1-F\left(x\right)= x^{-\alpha}L\left(x\right)$
with $0<\alpha<1$ and $L$ a \emph{slowly varying} function, i.e.\
for every $c>0$ we have $L\left(cx\right)\sim L\left(x\right)$.
Here, $u\left(x\right)\sim w\left(x\right)$ means that $u$ is
asymptotic to $w$ at infinity, i.e.\ $\lim_{x\rightarrow\infty}u\left(x\right)/w\left(x\right)=1$.

Under certain properties of the underlying process to be discussed later
an \emph{intermediately trimmed strong law} has been proven for such observables,
i.e.\ there exist a  non-negative integer sequence $\left(b_n\right)$
tending to infinity with $b_n=o(n)$ and a norming sequence $\left(d_n\right)$ such that 
\begin{align*}
 \lim_{n\to\infty}\frac{\mathsf{S}_n^{b_n}\chi}{d_n}=1\text{ a.s.}
\end{align*}
Additionally, an asymptotic formula for $\left(d_n\right)$ depending on $\left(b_n\right)$ has been provided in \cite[Theorem 1.7]{kessebohmer_strong_2018}.
The condition under which these trimming results hold are in particular a spectral gap property 
for the transfer operator und some regularity conditions on the observable $\chi$, 
the precise conditions are stated as Property $\mathfrak{D}$ in Definition \ref{def: Prop D}.

The above stated intermediately trimmed strong law can be seen as an analog to Birkhoff's ergodic theorem.
However in the finite case, Birkhoff's ergodic theorem also implies
that the norming sequence $(d_n)$ can be chosen as $(\int \mathsf{S}_n\chi\;\mathrm{d}\mu)$, see e.g.\ \cite[Prop.\ 2.4.21]{MR3585883}.
It is the purpose of the present paper to show that for regularly varying tail distributions 
with exponent strictly between $-1$ and $0$  we also have
$d_n\sim \int\mathsf{S}_n^{b_n}\;\mathrm{d}\mu$ 
and  in this way to give an analog statement of the ergodic theorem for the trimmed sum process.
Since  in Lemma \ref{lem: conv prob} we also  show convergence  in probability,  the above asymptotic is  in fact equivalent to mean convergence (Theorem \ref{thm: conv in mean})
  and gives thus also an analog for von Neumann's $\mathcal{L}^{1}$ ergodic theorem.
Crucial for our analysis will be an additional condition on $\left(\chi\circ T^{n-1}\right)$ 
 in terms of the  $\bm{\psi}$-mixing coefficients 
 introduced in Definition \ref{def: phi psi mixing}.
We will show that  Property $\mathfrak{D}$  alone is indeed not strong enough for our main results to hold
by providing an example with large  $\bm{\psi}$-mixing coefficients and which 
can not obey any mean convergence, as $\int\mathsf{S}_n^{b_n}\;\mathrm{d}\mu=\infty$, 
for all $n\in\mathbb{N}$ and any reasonable trimming sequence $(b_{n})$, see Theorem \ref{thm: counterex}.

In the case of general distribution functions and the same mixing conditions it is too much to hope 
for such a mean convergence under trimming.
The authors of this paper gave an example in \cite[Remark 3]{kessebohmer_strong_2016} for i.i.d.\ random variables
for which an intermediately trimmed strong law holds but $\int\mathrm{S}_n^{b_n}\;\mathrm{d}\mu=\infty$, 
for all $n\in\mathbb{N}$.

It is also worth mentioning that the almost sure trimming results mentioned above have predecessors in results for i.i.d.\ random variables
where a vast literature for trimming results both for weak as well as for strong limit theorems exists.
However, to the author's knowledge the result given in this paper has not been proven for i.i.d.\ 
random variables either.
We will here only   give an overview of a number of  strong convergence results.
First of all one realizes that also for i.i.d.\ random variables a
\emph{lightly trimmed strong law} can not hold   for random variables with regularly varying tail with exponent strictly between $-1$ and $0$.
By a lightly trimmed strong law we mean  the existence of $r\in\mathbb{N}$ and a sequence $(d_n)$   of positive reals 
such that $\lim_{n\to\infty}S_n^r\chi/d_n=1$ a.s. 
This can be deduced from the fact that there is no weak law of large numbers for random variables 
with such a distribution function, see \cite[VII.7 Theorem 2 and VIII.9 Theorem 1]{feller_introduction_1971}
and a result by Kesten which states that light trimming has no influence on weak convergence, see \cite{kesten_convergence_1993}.
  However, an  intermediately trimmed strong law in the i.i.d.\ case can be deduced   from  results
by Haeusler and Mason, 
see \cite{haeusler_laws_1987} and \cite{haeusler_nonstandard_1993}. 
They proved generalized laws of the iterated logarithm under trimming from which an intermediately trimmed strong law follows 
and one can also infer a lower bound for $\left(b_n\right)$. 
  Indeed, this lower bound coincides with the lower bound for $(b_n)$ in the dynamical systems case 
given in \cite{kessebohmer_strong_2018}. 
 The examples for which the setting of \cite{kessebohmer_strong_2018}
holds are e.g.\ piecewise expanding interval maps and subshifts of finite type 
as shown in \cite{kessebohmer_strong_2018} and \cite{kessebohmer_intermediately_2019} respectively.
 
Some  other  trimming results have also been generalized to different dynamical system settings where $\bm{\psi}$-mixing also plays an important role. 
In fact, Aaronson and Nakada showed in \cite{aaronson_trimmed_2003} a lightly trimmed strong law 
for $\bm{\psi}$-mixing random variables which have particular distribution functions.
This result generalizes the results of Mori for the i.i.d.\ case, see \cite{mori_strong_1976,mori_stability_1977}.
  Aaronson and Nakada  also gave an example of a non $\bm{\psi}$-mixing process with the same distribution 
function not fulfilling a lightly trimmed strong law.
Furthermore, Haynes gave in \cite{haynes_quantitative_2014} a quantitative generalization for $\bm{\psi}$-mixing random variables 
of a result by Diamond and Vaaler who showed a lightly trimmed strong
law for the continued fraction digits, see \cite{diamond_estimates_1986}.
 Haynes also compared this result with an observable on the doubling map for which the system is strongly mixing 
but not $\bm{\psi}$-mixing and for which a lightly trimmed strong laws   fails to  hold. 
This example also fulfills the spectral gap property and Property $\mathfrak{D}$. 

The results of this paper rely on two main properties:
 First, on an exponential inequality for dynamical systems 
which fulfill a spectral gap property with respect to the transfer operator,
and second, on the $\bm{\psi}$-mixing property.
The proof of  the exponential inequality  given in \cite{kessebohmer_strong_2018} is similar to the Nagaev-Guivarc'h spectral method for the central limit theorem.
The spectral gap property for dynamical systems 
is a typical assumption under which limit theorems for dynamical systems can be proven, 
see the review papers \cite{gouezel_limit_2015} and \cite{fan_spectral_2003} and references therein
for further information and applications of the transfer operator method 
as well as examples of dynamical systems fulfilling a spectral gap property.

During the last decade there has also been some  significant interest 
in other limit theorems for dynamical systems with heavy tailed distributions using transfer operator techniques,
particularly convergence to a stable law, see the paper by Aaronson and Denker, \cite{aaronson_local_2001},
for sufficient and the paper by Gou\"ezel, \cite{gouezel_characterization_2010}, for necessary conditions 
and previous results by Sarig, \cite{sarig_continuous_2006}.
Furthermore, see also the generalization by Melbourne and Zweim\"uller to intermittent maps,
\cite{melbourne_weak_2015}, and by Tyran-Kaminska to a functional convergence, see \cite{tyran_weak_2010}.
The additional condition of  $\bm{\psi}$-mixing  
is often necessary for proving  limit theorems  as illustrated 
above for the trimmed strong laws. Some results have also been proven under 
the combined assumptions of a spectral gap property 
and $\bm{\psi}$-mixing, as for example the law of an iterated logarithm 
for non-integrable random variables by Aaronson and Zweim\"uller, \cite{aaronson_limit_2014}.
 
In Example \ref{ex: interval maps} 
we will give conditions on piecewise expanding interval maps and on the observable $\chi$ such that  
Property $\mathfrak{D}$ as well as the $\bm{\psi}$-mixing condition are fulfilled.
 
\subsection{Basic setting}\label{subsec: setting def}
First we will make the notion of  spectral gap precise and then restate the two crucial properties from \cite{kessebohmer_strong_2018}. 
The first, Property $\mathfrak{C}$, considers dynamical systems with a spectral gap property.
Afterwards we define our main property, Property $\mathfrak{D}$, for which different convergence theorems 
have been proven in \cite{kessebohmer_strong_2018} and under which we will prove  a mean convergence theorem under trimming.

\begin{Def}[Spectral gap]\label{def spec gap}
Suppose $\mathcal{F}$ is a Banach space and $U:\mathcal{F}\to\mathcal{F}$ a bounded linear operator. 
We say that $U$ has a \emph{spectral gap} if there exists a decomposition $U=\lambda P+N$ with $\lambda\in\mathbb{C}$ and $P,N$ bounded linear operators such that 
\begin{itemize}
\item $P$ is a one-dimensional projection, i.e.\ $P^2=P$ and its image is one-dimensional,
\item $N$ is such that $\rho\left(N\right)<\left|\lambda\right|$,  where $\rho$ denotes the spectral radius,
\item $P$ and $N$ are orthogonal, i.e.\ $PN=NP=0$\index{orthogonal}.
\end{itemize}
\end{Def}

\begin{Def}[Property $\mathfrak{C}$, {\cite[Definition 1.1]{kessebohmer_strong_2018}}]\label{def: prop C}
Let $\left(\Omega, \mathcal{A},  T, \mu\right)$ be a dynamical system with $ T$ a non-singular transformation and $\widehat{ T}:\mathcal{L}^1\to \mathcal{L}^1$ be the transfer operator of $ T$, 
i.e.\ the uniquely defined operator such that for all $f\in\mathcal{L}^1$ and $g\in\mathcal{L}^{\infty}$ we have
\begin{align}
\int \widehat{ T}f\cdot g\;\mathrm{d}\mu=\int f\cdot g\circ T\;\mathrm{d}\mu,\label{hat CYRI}
\end{align}
see e.g. \cite[Section 2.3]{MR3585883} for further details.
Furthermore, let $\mathcal{F}$ be subset of the measurable functions forming a Banach algebra with respect to the  norm $\left\|\cdot\right\|$.
We say that $\left(\Omega, \mathcal{A}, T , \mu,\mathcal{F},\left\|\cdot\right\|\right)$ has Property $\mathfrak{C}$ if the following conditions hold:
\begin{itemize}
\item
$\mu$ is a $T $-invariant, mixing probability measure.
\item
$\mathcal{F}$ contains the constant functions and for all $f\in\mathcal{F}$ we have
\begin{align}
\left\|f\right\|\geq \left|f\right|_{\infty}.\label{ineq <l}
\end{align}
\item 
$\widehat{T}$ is a bounded linear operator with respect to $\left\|\cdot\right\|$, i.e.\ there {exists} a constant $K_0>0$ such that for all $f\in\mathcal{F}$ we have
\begin{align}
\left\|\widehat{T}f\right\|\leq K_0\cdot\left\|f\right\|.\label{C 0 f}
\end{align}
\item 
$\widehat{T}$ acting on the Banach space  $\mathcal{F}$ with norm $\left\|\cdot\right\|$ has a spectral gap.
\end{itemize}
\end{Def}
The above mentioned property is a widely used setting for dynamical systems. 
In particular it implies that the transfer operator has $1$ 
as a unique and simple eigenvalue on the unit circle and that an exponential decay of correlation is guaranteed.
 
However, in order to state our main theorems we need additional assumptions on the observable $\chi$ defined on a system fulfilling Property $\mathfrak{C}$.
\begin{Def}[Property $\mathfrak{D}$, {\cite[Definition 1.2]{kessebohmer_strong_2018}}]\label{def: Prop D}
We say that $\left(\Omega, \mathcal{A}, T , \mu,\mathcal{F},\left\|\cdot\right\|,\chi\right)$ has Property $\mathfrak{D}$ if the following conditions hold:
\begin{itemize}
\item $\left(\Omega, \mathcal{A}, T , \mu,\mathcal{F},\left\|\cdot\right\|\right)$ fulfills Property $\mathfrak{C}$.
\item $\chi\in \mathcal{M}(\mathcal{A})^{+}$ and  
 with $\upperleft{\ell}{}{\chi}\coloneqq \chi\cdot\mathbbm{1}_{\left\{\chi\leq \ell\right\}}$ there exists
$K_{1}>0$ such that for all $\ell\in\mathbb{R}_{\geq 0}$, 
\begin{align}
\left\|\upperleft{\ell}{}{\chi}\right\|\leq K_{1}\cdot \ell
\;\;\;\; \mbox{ and }\;\;\;\;
\left\|\mathbbm{1}_{\left\{\chi>\ell\right\}}\right\|\leq K_{1}.\label{C 1} 
\end{align}

\end{itemize}
\end{Def}

Finally, to state our main theorem we give the precise definition of $\bm{\psi}$-mixing following \cite{bradley_basic_2005}. Note that in the literature there are sometimes subtle differences defining this notion.
\begin{Def}\label{def: phi psi mixing}
Let $\left(\Omega,\mathcal{A},\mathbb{P}\right)$ be a probability measure space and $\mathcal{B},\mathcal{C}\subset\mathcal{A}$ two $\sigma$-fields, then the following measure of dependence is defined
\begin{align*}
\bm{\psi}\left(\mathcal{B},\mathcal{C}\right)&\coloneqq \sup\left\{\left|\frac{\mathbb{P}\left(B\cap C\right)}{\mathbb{P}\left(B\right)\cdot \mathbb{P}\left(C\right)}-1\right|\colon B\in\mathcal{B}, C\in\mathcal{C}, \mathbb{P}\left(B\right), \mathbb{P}\left(C\right)>0\right\}.
\end{align*}

Furthermore, let $\left(X_{n}\right)_{n\in\mathbb{N}}$ be a (not necessarily stationary) sequence of random variables. For $-\infty\leq J\leq L\leq\infty$ we can define a $\sigma$-field by
\begin{align*}
\mathcal{A}_{J}^{L}\coloneqq \sigma\left(X_{k},J\leq k\leq L, k\in\mathbb{Z}\right).
\end{align*}
With this at hand the  {\em $\bm{\psi}$-mixing  coefficients} are defined by
\begin{align*}
\bm{\psi}\left(n\right)&\coloneqq \sup_{k\in\mathbb{N}}\bm{\psi}\left(\mathcal{A}_{-\infty}^{k},\mathcal{A}_{k+n}^{\infty}\right).
\end{align*}
The sequence of random variables $\left(X_{n}\right)$ is said to be \emph{$\bm{\psi}$-mixing}   if  $\lim_{n\to\infty}\bm{\psi}\left(n\right)=0$.
\end{Def}

\subsection{Main results}\label{subsec: main results}
For $L$ being slowly varying we denote by $L^{\#}$
a de Bruijn conjugate of $L$, i.e.\ 
a slowly varying function satisfying 
\begin{align*}
\lim_{x\rightarrow\infty}L\left(x\right)\cdot L^{\#}\left(xL\left(x\right)\right)  =1=
\lim_{x\rightarrow\infty}L^{\#}\left(x\right)\cdot L\left(xL^{\#}\left(x\right)\right).
\end{align*}
For more details  see \cite[Section 1.5.7 and Appendix 5]{bingham_regular_1987}.
Then our  first main result reads as follows.

\begin{thm}\label{thm: conv in mean}
Let $\left(\Omega, \mathcal{A}, T , \mu,\mathcal{F},\left\|\cdot\right\|,\chi\right)$ fulfill Property $\mathfrak{D}$
and assume that $\mu\left(\chi>x\right)=L\left(x\right)/x^{\alpha}$, where $L$ is a slowly varying function and $0<\alpha<1$.
Further, let $\left(b_n\right)$ be a sequence of natural numbers tending to infinity such that $b_n=o\left(n\right)$.
We assume that at least one $\bm{\psi}$-mixing coefficient  of the sequence of random variables $\left(\chi\circ T^{n-1}\right)_{n\in\mathbb{N}}$ is strictly less than one.
Then mean convergence holds with norming sequence 
\begin{align}
d_n\coloneqq
\frac{\alpha}{1-\alpha}\cdot n^{1/\alpha}\cdot b_{n}^{1-1/\alpha}\cdot \left(L^{-1/\alpha}\right)^{\#}\left(\left(\frac{n}{b_{n}}\right)^{1/\alpha}\right),\label{eq: def dn}
\end{align}
that is
\begin{align*}
\lim_{n\to\infty}\int\left|\frac{\mathsf{S}_n^{b_n}\chi}{d_n}-1\right|\,\mathrm{d}\mu=0.
\end{align*}
\end{thm}

\begin{Rem}
It has been shown in \cite{kessebohmer_strong_2018} that under Property $\mathfrak{D}$ and with $(b_n)$ growing sufficiently fast we have   $\lim_{n\to\infty}\mathsf{S}_n^{b_n}\chi/d_n=1$ a.s.\ where $(d_n)$ shows the same asymptotic 
as \eqref{eq: def dn}.
Combining this statement with Theorem \ref{thm: conv in mean} 
yields that for sufficiently fast growing $(b_n)$ we have 
\[\lim_{n\to\infty}\frac{\mathsf{S}_n^{b_n}\chi}{\int \mathsf{S}_n^{b_n}\chi\;\mathrm{d}\mu}=1\text{ a.s.}
 \]
\end{Rem}
\begin{Rem}
 Note that  $\bm{\psi}$-mixing is sufficient but not necessary    for  Theorem \ref{thm: conv in mean}. 
\end{Rem}

For our main example, we first define the space of functions of bounded variation. 
For simplicity, we restrict ourself to the interval $\left[0,1\right]$ and let $\mathcal{B}$ denote the Borel sets of $\left[0,1\right]$. 
\begin{Def}\label{def: bound var}
Let $\varphi:\left[0,1\right]\to\mathbb{R}$. Then the variation $\mathsf{var}\left(\varphi\right)$ of $\varphi$ is given by 
\begin{align*}
\mathsf{var}\left(\varphi\right)\coloneqq\sup\left\{\sum_{i=1}^n\left|\varphi\left(x_i\right)-\varphi\left(x_{i-1}\right)\right|\colon n\geq 1, x_i\in\left[0,1\right], x_0<x_1<\ldots<x_n\right\}
\end{align*}
and we define
\begin{align*}
\mathsf{V}\left(\varphi\right)\coloneqq  \inf\left\{\mathsf{var}\left(\varphi'\right)\colon \varphi'\text{ is a version of }\varphi\right\}.
\end{align*}
By $BV$ we denote the Banach space of functions of bounded variation, i.e.\ of functions $\varphi$ 
fulfilling $\mathsf{V}\left(\varphi\right)<\infty$. It is equipped with the norm
$\left\|\varphi\right\|_{BV}\coloneqq\left|\varphi\right|_{\infty}+\mathsf{V}\left(\varphi\right)$.
\end{Def}
With this we can state our main example. 
\begin{Ex}\label{ex: interval maps}
Let $\Omega'\subset \left[0,1\right]$ be a dense and open set such that $\mu\left(\Omega'\right)=1$ and
let $\mathcal{I}\coloneqq\left(I_n\right)_{n\in\mathbb{N}}$ be a countable family
of closed intervals with disjoint interiors and for any $I_n$ such that the set $I_n\cap \left(\left[0,1\right]\backslash \Omega'\right)$ 
consists exactly of the endpoints of $I_n$.
Furthermore, we assume that $T$ fulfills the following properties:
\begin{itemize}
 \item (Adler's condition) $ T_n\coloneqq T\lvert_{\mathring{I}_n}\in \mathcal{C}^2$ and
 $ T''/\left( T'\right)^2$ is bounded on $\Omega'$.
 \item (Finite image condition) $\#\left\{ T I_n\colon I_n\in\mathcal{I}\right\}<\infty$.
 \item (Uniform expansion)
 There exists $m>1$ such that  $\left| T_n'\right|\geq m$ for all $n\in\mathbb{N}$.
 \item $ T$ is topologically mixing. 
\end{itemize}
Furthermore, we assume that $\chi$ is constant on the interior of each interval $I_n$
and there exists a constant $k >0$ such that 
for all $\ell\in\mathbb{R}_{\geq 0}$
\begin{align}
\mathsf{V}\left(\upperleft{\ell}{}{\chi}\right)\leq k\cdot {\ell}
\hspace{2cm}\text{ and }\hspace{2cm}
\mathsf{V}\left(\mathbbm{1}_{\left\{\chi>\ell\right\}}\right)\leq k.\label{eq: cond C 2}
\end{align}
Then there exists a probability measure $\mu$ absolutely continuous to the Lebesgue measure
such that $\left([0,1], \mathcal{B}, T, \mu, BV,\left\|\cdot\right\|_{BV}, \chi\right)$
fulfills Property $\mathfrak{D}$ and $\left(\chi\circ T^{n-1}\right)$ is a $\bm{\psi}$-mixing.

We note here that this example mainly relies on results in \cite{rychlik_bounded_1983} 
on piecewise expanding interval maps on countable partitions generalizing \cite{lasota_existence_1973} where  finite partitions are considered.
It follows from \cite{zweimueller_ergodic_1998} 
that $\left([0,1], \mathcal{B}, T, \mu, BV, \left\|\cdot\right\|_{BV}\right)$ fulfills Property $\mathfrak{C}$.
Furthermore, \eqref{eq: cond C 2} implies \eqref{C 1}   and thus $\left([0,1], \mathcal{B}, T, \mu, BV, \left\|\cdot\right\|_{BV}\right)$ 
fulfills Property $\mathfrak{D}$.
This was discussed in detail in \cite[Section 1.4]{kessebohmer_strong_2018}. 
Finally, that $\left(\chi\circ T^{n-1}\right)$ is $\bm{\psi}$-mixing follows by 
\cite[Theorem 1]{aaronson_mixing_2005}.
As an  explicit example one could choose the partition $I_n\coloneqq [1/(n+1), 1/n]$, $n\in \N$,
with any $T$ fulfilling the above properties with respect to this partition and 
$\chi=\sum _{n\in \N}n^{1/\alpha}\cdot\1_{I_{n}}$  and the underlying invariant measure calculated as in \cite{rychlik_bounded_1983}.
\end{Ex}

Our next theorem shows that indeed Property $\mathfrak{D}$ alone is not sufficient for our main theorem.
\begin{thm}\label{thm: counterex}
 Let $\Omega\coloneqq [0,1)$ and $T\coloneqq 2x\mod 1$ with $\mu$ the Lebesgue measure restricted to $[0,1)$. 
 Further define $\chi\colon [0,1)\to\mathbb{R}_{>0}$ by $\chi\left(x\right)=x^{-\gamma}$ with $\gamma>1$. 
 Then there exists a Banach space $\mathcal{F}$ with a norm $\left\|\cdot\right\|$ such that 
 $\left([0,1),\mathcal{B},T,\mu, \mathcal{F},\left\|\cdot\right\|,\chi\right)$ fulfills Property $\mathfrak{D}$.
 If on the other hand $(b_n)$ tends to infinity with $b_n=o(n)$,
 then 
\begin{align*}
 \int\mathsf{S}_n^{b_n}\chi\;\mathrm{d}\mu=\infty,
\end{align*}
for all $n\in\mathbb{N}$.
\end{thm}

\begin{Rem}\label{rem: iid special case}
To the authors' knowledge this theorem is also a new result for the setting of i.i.d.\ random variables. 
Indeed the case of i.i.d.\ random variables follows as a special case from this setting. 
A proof of this fact  will be given in Section \ref{subsec: proof remark}.
\end{Rem}

\section{Proofs of main theorems}\label{sec: proof main thm}
The second part of our paper is devoted to the  proofs of the theorems;  
the proof of Theorem \ref{thm: conv in mean} will be given in Section \ref{subsec: Proof mean conv},
 the proof of Theorem \ref{thm: counterex} in Section \ref{subsec: proof thm counterex}, and
the proof of Remark \ref{rem: iid special case} in Section \ref{subsec: proof remark}.

\subsection{Proof of Theorem \ref{thm: conv in mean}}\label{subsec: Proof mean conv}
Theorem \ref{thm: conv in mean} is proven by proving the following two lemmas: 
\begin{lemma}\label{lem: mean assymp}
Let $\left(\Omega, \mathcal{A}, T , \mu,\mathcal{F},\left\|\cdot\right\|,\chi\right)$ fulfill Property $\mathfrak{D}$
and assume that $\mu\left(\chi>x\right)=L\left(x\right)/x^{\alpha}$, where $L$ is a slowly varying function and $0<\alpha<1$.
Further, let $\left(b_n\right)$ be a sequence of natural numbers tending to infinity such that $b_n=o\left(n\right)$.
We assume that at least one $\bm{\psi}$-mixing coefficient of the sequence of random variables $\left(\chi\circ T^{n-1}\right)_{n\in\mathbb{N}}$ is strictly less than one.
Then  
$\int{\mathsf{S}_n^{b_n}\chi}\sim d_n$
with $(d_n)$ as in \eqref{eq: def dn}. 
\end{lemma}

\begin{lemma}\label{lem: conv prob}
Let $\left(\Omega, \mathcal{A}, T , \mu,\mathcal{F},\left\|\cdot\right\|,\chi\right)$ fulfill Property $\mathfrak{D}$
and assume that $\mu\left(\chi>x\right)=L\left(x\right)/x^{\alpha}$, where $L$ is a slowly varying function and $0<\alpha<1$.
Further, let $\left(b_n\right)$ be a sequence of natural numbers tending to infinity such that $b_n=o\left(n\right)$.
Then  with $(d_n)$ given in \eqref{eq: def dn}, we have the following convergence in probability:
\begin{align*}
\lim_{n\to\infty}\mu\left(\left|\frac{\mathsf{S}_n^{b_n}\chi}{d_n}-1\right|>\epsilon\right)=0,\;\;\mbox{ for all } \epsilon >0. 
\end{align*} 
\end{lemma}

Using Pratt's theorem  \cite{MR0123673} in combination with Lemma \ref{lem: mean assymp} and \ref{lem: conv prob} immediately gives the statement of Theorem \ref{thm: conv in mean}.

In preparation of  the proof of these lemmas,
for $\chi:\Omega\to\mathbb{R}_{\geq 0}$ 
and a real valued sequence $\left(f_{n}\right)_{n\in\mathbb{N}}$
we recall the definition of the truncated function 
\begin{align*}
 \upperleft{f_n}{}{\chi}\coloneqq \chi\cdot\mathbbm{1}_{\left\{\chi\leq f_n\right\}}
\end{align*}
 given in Definition \ref{def: Prop D} 
and define the associated truncated sum 
\begin{align*}
 \mathsf{T}_n^{f_n}\chi\coloneqq \sum_{k=1}^n\upperleft{f_n}{}{\chi}\circ T^{k-1}.
\end{align*}
If $f_n$ tends to infinity, we have that 
\begin{align}
 \int\mathsf{T}_n^{f_n}\chi\;\mathrm{d}\mu
 &\sim n\cdot\frac{\alpha}{1-\alpha}\cdot L\left(f_n\right)\cdot f_n^{1-\alpha},\label{E S*}
\end{align}
see \cite[Lemma 3.18]{kessebohmer_strong_2018} for a detailed calculation.

\begin{proof}[Proof of Lemma \ref{lem: mean assymp}]
We recall that $F$ is the distribution function of $\chi$ with respect to $\mu$, i.e.\ 
$F(x)=\mu\left(\chi\leq x\right)=1-L\left(x\right)/x^{\alpha}$.
Let $\left(\zeta_n\right)$ be defined as $\zeta_n\coloneqq b_n^{2/3}$
and set 
\begin{align}
 g_n\coloneqq F^{\leftarrow}\left(1-\frac{b_n-\zeta_n}{n}\right).\label{eq: un}
\end{align}

We will split the proof of the  lemma into the following parts:
\begin{enumerate}[label=(\Alph*)]
\item\label{en: C}
We have that
\begin{align}
\int\mathsf{T}_{n}^{g_{n}}\chi\;\mathrm{d}\mu\sim\frac{\alpha}{1-\alpha}\cdot n^{1/\alpha}\cdot b_{n}^{1-1/\alpha}\cdot \left(L^{1/\alpha}\right)^{\#}\left(\left(\frac{n}{b_{n}}\right)^{1/\alpha}\right).\label{eq: asymptotic in A}
\end{align}
\item\label{en: A} For all $\epsilon>0$
there exists $N\in\mathbb{N}$ such that for all $n\geq N$ 
\begin{align*}
\int \mathsf{S}_{n}^{b_{n}}\chi\;\mathrm{d}\mu\leq\left(1+\epsilon\right)\int\mathsf{T}_{n}^{g_n}\chi\;\mathrm{d}\mu.\end{align*}
\item\label{en: B}
For all $\epsilon>0$ there exists
$N\in\mathbb{N}$ such that for all $n\geq N$ 
\begin{align*}
\int \mathsf{S}_{n}^{b_{n}}\chi\;\mathrm{d}\mu\geq\left(1-\epsilon\right)\int\mathsf{T}_{n}^{g_n}\chi\;\mathrm{d}\mu. \end{align*}
\end{enumerate}

\emph{Proof of \ref{en: C}}:
The proof of a statement similar to \ref{en: C} can be found  at the end of the proof of Theorem 1.7 in \cite{kessebohmer_strong_2018}. 
Indeed, there it is shown that
$\int\mathsf{T}_{n}^{f_{n}}\chi\;\mathrm{d}\mu$ is asymptotic to the right hand side of \eqref{eq: asymptotic in A}.
 The sequence $(f_n)$ does not necessarily coincide with $(g_n)$.
However, the sequences can be written as 
$f_n= F^{\leftarrow}\left(1-v_n/n\right)$
and 
$g_n= F^{\leftarrow}\left(1-w_n/n\right)$
with $v_n\sim u_n\sim b_n$.
The proof in \cite{kessebohmer_strong_2018} essentially proves that 
$\int\mathsf{T}_{n}^{f_{n}}\chi\;\mathrm{d}\mu\sim \alpha/(1-\alpha)\cdot n^{1/\alpha}\cdot v_{n}^{1-1/\alpha}\cdot \left(L^{-1/\alpha}\right)^{\#}\big(\left(n/v_{n}\right)^{1/\alpha}\big)$
allowing us to   conclude from $v_n\sim b_n$ the asymptotic in \eqref{eq: asymptotic in A}. The sequence $(g_n)$ can be treated analogously.
\medskip

\emph{Proof of \ref{en: A}}:
In order to prove \ref{en: A} we set for $k,n\in\mathbb{N}$, 
\begin{align}
\Gamma_{n} & \coloneqq\left\{ \mathsf{S}_{n}^{b_{n}}\chi\leq \mathsf{T}_{n}^{g_n}\chi\right\} ,\notag\\
\Delta_{k,n} & \coloneqq\left\{ \mathsf{T}_{n}^{2^{k-1}\cdot g_n}\chi<\mathsf{S}_{n}^{b_{n}}\chi\leq \mathsf{T}_{n}^{2^{k}\cdot g_n}\chi\right\} ,\notag\\
\Phi_{k,n} & \coloneqq\left\{ \mathsf{T}_{n}^{2^{k-1}\cdot g_n}\chi<\mathsf{S}_{n}^{b_{n}}\chi\right\} .\label{eq: def Phi kn}
\end{align}
Clearly,
\begin{align}
\int \mathsf{S}_{n}^{b_{n}}\chi\;\mathrm{d}\mu & \leq\int \mathsf{T}_{n}^{g_n}\chi\cdot\mathbbm{1}_{\Gamma_{n}}\;\mathrm{d}\mu+\sum_{k=1}^{\infty}\int \mathsf{T}_{n}^{2^{k}\cdot g_n}\chi\cdot\mathbbm{1}_{\Delta_{k,n}}\;\mathrm{d}\mu\notag\\
  &\leq\int \mathsf{T}_{n}^{g_n}\chi\;\mathrm{d}\mu+\sum_{k=1}^{\infty}\int \mathsf{T}_{n}^{2^{k}\cdot g_n}\chi\cdot\mathbbm{1}_{\Phi_{k,n}}\;\mathrm{d}\mu.\label{int Tnfn}
\end{align}

We will show in the following that $\sum_{k=1}^{\infty}\int \mathsf{T}_{n}^{2^{k}\cdot g_n}\chi\cdot\mathbbm{1}_{\Phi_{k,n}}\;\mathrm{d}\mu$
is negligible compared to $\int \mathsf{T}_{n}^{g_n}\chi\;\mathrm{d}\mu$.
First we define $r\coloneqq \min\left\{n\in\mathbb{N}\colon \bm{\psi}\left(n\right)<1\right\}$ and
\begin{align*}
\Phi_{k,n}'\coloneqq\left\{\# \left\{j\leq n\colon \chi\circ T^{j-1} >2^{k-1}\cdot g_n\right\}> b_n-2r+1\right\}.
\end{align*}
Since $\Phi_{k,n}= \left\{\# \left\{j\leq n\colon \chi\circ T^{j-1} >2^{k-1}\cdot g_n\right\}> b_n\right\}$,
we have that $\Phi_{k,n}\subset \Phi_{k,n}'$.
We will show \ref{en: A} by proving the following three statements:
\begin{enumerate}[label=(B\arabic*)]
 \item\label{en: A1}
 We have for all $k\geq 1$ and $n$ sufficiently large uniformly in $k$ that
 \begin{align*}
  \int \mathsf{T}_{n}^{2^{k}\cdot g_n}\chi\cdot\mathbbm{1}_{\Phi_{k,n}}\;\mathrm{d}\mu
  \leq \frac{\left(1+\bm{\psi}\left(r\right)\right)^2}{1-{\bm \psi}\left(r\right)}\cdot\int\mathsf{T}_{n}^{2^k\cdot g_n}\chi\;\mathrm{d}\mu
  \cdot \mu\left(\Phi_{k,n}'\right). 
   \end{align*}
 \item\label{en: A2}
 We have
 \begin{align*} 
\sum_{k=2}^{\infty}\int \mathsf{T}_{n}^{2^{k}\cdot g_n}\chi\;\mathrm{d}\mu\cdot\mu\left(\Phi_{k,n}'\right) &=o\left(\int \mathsf{T}_{n}^{g_n}\chi\;\mathrm{d}\mu\right).
\end{align*}
 \item\label{en: A3} We have
 \begin{align*}
  \int \mathsf{T}_{n}^{2\cdot g_n}\chi\;\mathrm{d}\mu\cdot\mu\left(\Phi_{1,n}'\right)=o\left(\int \mathsf{T}_{n}^{g_n}\chi\;\mathrm{d}\mu\right).
 \end{align*}
\end{enumerate}
Combining these statements with \eqref{int Tnfn} proves \ref{en: A}.
In \ref{en: A1} we have used a short notation which we will also use in the sequel. 
If we write that a statement $A_{k,n}$ depending on $n$ and $k$ holds for $n$ sufficiently large uniformly in $k$
we mean that there exists $N\in\mathbb{N}$ such that the statement $A_{n,k}$ holds 
for all $n\geq N$ and all $k$.

\emph{Proof of \ref{en: A1}}:
We will start this section with a set of definitions explaining in the sequel the strategy of the proof.
Let 
\begin{align*}
 E_{k,n,m,1}\coloneqq\{\chi\circ T^{m-1}\leq 2^{k-1}\cdot g_n\}
 \text{ and }E_{k,n,m,2}\coloneqq\{\chi\circ T^{m-1}> 2^{k-1}\cdot g_n\}
 \text{ and }E_{k,n,m,3}\coloneqq\Omega.
\end{align*}
Further, for $J\in\{1,2,3\}^n$ let
$D_{k,n,J}\coloneqq\bigcap_{J=(j_m)_{1\leq m\leq n}}E_{k,n,m,j_m}$.
Loosely speaking our set $D_{k,n,J}$ determines for each $m\leq n$ if $\chi\circ T^{m-1}>2^{k-1}\cdot g_n$ or $\chi\circ T^{m-1}\leq 2^{k-1}\cdot g_n$ holds,
or if no information about $\chi\circ T^{m-1}$ is gained.

 Remember that $r= \min\left\{n\in\mathbb{N}\colon \bm{\psi}\left(n\right)<1\right\}$.
Further, let $\Gamma_{n,i}\coloneqq \left\{m\in\mathbb{N}_{\leq n}\colon \left|m-i\right|\geq r\right\}$ and
let 
\begin{align*}
 \mathcal{J}_{k,n}^{i}&\coloneqq\left\{J=(j_m)_{1\leq m\leq n}\in\left\{1,2,3\right\}^n\colon j_m\neq 3\text{ for }m\in \Gamma_{n,i}\text{ and }j_m=3\text{ for }m\in (\Gamma_{n,i})^c\right.\\
 &\qquad\left.\text{ and }\#\left\{m\in \Gamma_{n,i}\colon \chi\circ T^{m-1}>2^{k-1}\cdot g_n\right\}>b_n-2r+1\right\}.
\end{align*}
 
 Here and in the following we denote by $A^c$ the complement of a set $A$.
Since $\#\Gamma_{n,i}^c \leq 2r-1$, the definition of $\mathcal{J}_{k,n}^{i}$ implies for each $k,n\in\mathbb{N}$ and $i\in\mathbb{N}_{\leq n}$ that
\begin{align}
 \Phi_{k,n}\subset\biguplus_{J\in \mathcal{J}_{k,n}^i}D_{k,n,J}\subset \Phi_{k,n}'.\label{eq: Un bn1}
\end{align}
 Thus, we have for each $i\in\mathbb{N}_{\leq n}$ and $k,n\in\mathbb{N}$
\begin{align}
 \int \mathsf{T}_{n}^{2^{k}\cdot g_n}\chi\cdot\mathbbm{1}_{\Phi_{k,n}}\;\mathrm{d}\mu 
 &=\sum_{i=1}^{n}\int \upperleft{2^{k}\cdot g_n}{}{\chi}\circ T^{i-1}\cdot\mathbbm{1}_{\Phi_{k,n}}\;\mathrm{d}\mu
  \leq\sum_{i=1}^{n}\sum_{J\in \mathcal{J}_{k,n}^{i}}\int \upperleft{2^{k}\cdot g_n}{}{\chi}\circ T^{i-1}\cdot\mathbbm{1}_{D_{k,n,J}}\;\mathrm{d}\mu.\label{Un bn2}
\end{align}
 For each $k,n\in\mathbb{N}$, $i\in\mathbb{N}_{\leq n}$,
and $J\in\mathcal{J}^{i}_{k,n}$
we can write $D_{k,n,J}$ as an intersection of two events 
$D_{k,n,J}^{i,<}\cap  D_{k,n,J}^{i,>}$,
where
\begin{align*}
 D_{k,n,J}^{i,<}&\coloneqq \bigcap_{(j_m)=J, m=1,\ldots, i-r}E_{k,n,m,j_m}
 \,\,\,\,\,\,\,\,\,\text{ and }\,\,\,\,\,\,\,\,\,
 D_{k,n,J}^{i,>}\coloneqq \bigcap_{(j_m)=J, m=i+r,\ldots, n}E_{k,n,m,j_m}.
\end{align*}
If $i-r<1$, then we set $D_{k,n,J}^{i,<}\coloneqq \Omega$ and if $i+r>n$, then we set $D_{k,n,J}^{i,>}\coloneqq \Omega$.
Loosely speaking, $D_{k,n,J}^{i,<}$ is determined by
the information of the first $i-r$ entries and 
$D_{k,n,J}^{i,>}$ of the last entries starting from the $(i+r)$-th.
We notice that for $J\in\mathcal{J}_{n,k}^{i}$
\begin{align*}
 \int \upperleft{2^{k}\cdot g_n}{}{\chi}\circ T^{i-1}\cdot\mathbbm{1}_{D_{k,n,J}}\;\mathrm{d}\mu
 &= \int \mathbbm{1}_{D_{k,n,J}^{i,<}}\cdot \upperleft{2^{k}\cdot g_n}{}{\chi}\circ T^{i-1}\cdot\mathbbm{1}_{D_{k,n,J}^{i,>}}\;\mathrm{d}\mu.
\end{align*}
To estimate this term we will use the $r$th $\bm{\psi}$-mixing coefficient
for which by assumption we have that $\bm{\psi}(r)<1$.
To make use of the $\bm{\psi}$-mixing property we notice that for any random variables $X,Y$
\begin{align}
\mathrm{Cov}(X,Y)\leq \bm{\psi}\left(\sigma(X),\sigma(Y)\right)\cdot \left|X\right|_1\cdot \left|Y\right|_1,\label{eq: cov doukhan}
\end{align}
where we denote by $\mathrm{Cov}$  the covariance,
see e.g.\  \cite[Theorem 3, Chapter 1.2.2]{doukhan_mixing:_1994}.
  This implies for non-negative random variables 
\begin{align}
 \int X\cdot Y\;\mathrm{d}\mu\leq(1+\bm{\psi}\left(\sigma(X),\sigma(Y)\right))\cdot \int X\;\mathrm{d}\mu\cdot \int Y\;\mathrm{d}\mu.\label{eq: int XY estim psi}
\end{align}

The following statements will all hold for all $i\in\N_{\leq N}$, $k,n\in\N$ and $J\in\mathcal{J}_{k,n}^i$. 
For brevity we will not mention that for each of the following calculations.
By noticing that 
\begin{align*}
 \bm{\psi}\left(\sigma\left(\mathbbm{1}_{D_{k,n,J}^{i,<}}\right),\sigma\left(\upperleft{2^{k}\cdot g_n}{}{\chi}\circ T^{i-1}\cdot \mathbbm{1}_{D_{k,n,J}^{i,>}}\right)\right)&\leq \bm{\psi}(r)\text{ and }\\
 \bm{\psi}\left(\sigma\left(\upperleft{2^{k}\cdot g_n}{}{\chi}\circ T^{i-1}\right),\sigma\left(\mathbbm{1}_{D_{k,n,J}^{i,>}}\right)\right)&\leq \bm{\psi}(r)
\end{align*}
we obtain
\begin{align*}
 \MoveEqLeft\int \upperleft{2^{k}\cdot g_n}{}{\chi}\circ T^{i-1}\cdot\mathbbm{1}_{D_{k,n,J}}\;\mathrm{d}\mu\notag\\
 &= \int \mathbbm{1}_{D_{k,n,J}^{i,<}}\cdot \upperleft{2^{k}\cdot g_n}{}{\chi}\circ T^{i-1}\cdot\mathbbm{1}_{D_{k,n,J}^{i,>}}\;\mathrm{d}\mu\notag\\
 &\leq \left(1+\bm{\psi}\left(r\right)\right)\int \mathbbm{1}_{D_{k,n,J}^{i,<}}\;\mathrm{d}\mu \cdot \int\upperleft{2^{k}\cdot g_n}{}{\chi}\circ T^{i-1}\cdot\mathbbm{1}_{D_{k,n,J}^{i,>}}\;\mathrm{d}\mu\notag\\
 &\leq \left(1+\bm{\psi}\left(r\right)\right)\int \mathbbm{1}_{D_{k,n,J}^{i,<}}\;\mathrm{d}\mu \cdot \left(1+\bm{\psi}\left(r\right)\right) \int\upperleft{2^{k}\cdot g_n}{}{\chi}\circ T^{i-1}\;\mathrm{d}\mu\cdot\int\mathbbm{1}_{D_{k,n,J}^{i,>}}\;\mathrm{d}\mu\notag\\
 &= \left(1+\bm{\psi}\left(r\right)\right)^2\cdot \mu\left(D_{k,n,J}^{i,<}\right)\cdot \mu\left(D_{k,n,J}^{i,>}\right)\cdot\int\upperleft{2^{k}\cdot g_n}{}{\chi}\circ T^{i-1}\;\mathrm{d}\mu. 
 \end{align*}
Using \eqref{eq: cov doukhan} in the other direction gives for non-negative random variables
\begin{align*}
 \int X\mathrm{d}\mu\cdot \int X\mathrm{d}\mu\leq \frac{\int X\cdot Y\mathrm{d}\mu}{1-{\bm \psi}\left(\sigma\left(X\right),\sigma\left(Y\right)\right)}
\end{align*}
yielding
\begin{align*}
 \mu\left(D_{k,n,J}^{i,<}\right)\cdot \mu\left(D_{k,n,J}^{i,>}\right)
 &\leq \frac{\mu\left(D_{k,n,J}\right)}{1-{\bm \psi}\left(r\right)}
\end{align*}
and thus
\begin{align}
\int \upperleft{2^{k}\cdot g_n}{}{\chi}\circ T^{i-1}\cdot\mathbbm{1}_{D_{k,n,J}}\;\mathrm{d}\mu
 &\leq \frac{\left(1+\bm{\psi}\left(r\right)\right)^2}{1-{\bm \psi}\left(r\right)}\cdot \mu\left(D_{k,n,J}\right)\cdot \int\upperleft{2^{k}\cdot g_n}{}{\chi}\circ T^{i-1}\;\mathrm{d}\mu.\label{eq: int chi 1D < psi int chi int 1D}
\end{align}
Combining this with \eqref{eq: Un bn1} and \eqref{Un bn2} implies \ref{en: A1}.

\emph{Proof of \ref{en: A2}}:
We will first estimate $\int \mathsf{T}_{n}^{2^{k}\cdot g_n}\chi\;\mathrm{d}\mu$ using 
\eqref{E S*} and applying Potter's bound, see e.g.\ \cite[Theorem 1.5.6]{bingham_regular_1987}, which gives
\begin{align}
 \int \mathsf{T}_{n}^{2^{k}\cdot g_n}\chi\;\mathrm{d}\mu
 &\leq n\cdot2\cdot\frac{\alpha}{1-\alpha}\cdot L\left(2^k\cdot g_n\right)\cdot \left(2^k\cdot g_n\right)^{1-\alpha}\notag\\
 &\leq n\cdot2^{k+1}\cdot \frac{\alpha}{1-\alpha}\cdot L\left(g_n\right)\cdot g_n^{1-\alpha}\notag\\
 & \leq 2^{k+1}\cdot\int\mathsf{T}_{n}^{g_n}\chi\;\mathrm{d}\mu,\label{eq: Tn 2k Potter}
\end{align}
for $n$ sufficiently large uniformly in $k$.

Next, we will estimate $\mu\left(\Phi_{k,n}'\right)$.
We will use 
two different techniques, one for rather small and the other for larger $k$.
We start with the estimate for the smaller $k$.
To ease notation we set  $\overline{\mathbbm{1}}_A\coloneqq \mathbbm{1}_A-\mu\left(A\right)$, 
for any measurable set $A$.
We notice that for all  $k,n\in\N$,
\begin{align}
 \mu\left(\Phi_{k,n}'\right)
&=\mu\left(\sum_{i=1}^n\mathbbm{1}_{\left\{\chi\circ T^{i-1}>2^{k-1}\cdot g_n\right\}}> b_n-2r+1\right)\notag\\
&=\mu\left(\sum_{i=1}^n\overline{\mathbbm{1}}_{\left\{\chi\circ T^{i-1}>2^{k-1}\cdot g_n\right\}}> b_n-2r+1-n\cdot\mu\left(\chi>2^{k-1}\cdot g_n\right)\right)\notag\\
&\leq \mu\left(\left|\sum_{i=1}^n\overline{\mathbbm{1}}_{\left\{\chi\circ T^{i-1}>2^{k-1}\cdot g_n\right\}}\right|> b_n-2r+1-n\cdot\mu\left(\chi>2^{k-1}\cdot g_n\right)\right).\label{eq: Phi kn 1st estimate}
\end{align}
Furthermore, applying
the definition of the distribution function and  
Potter's bound implies 
\begin{align}
 n\cdot\mu\left(\chi>2^{k-1}\cdot g_n\right)
 = n\cdot\frac{L\left(2^{k-1}\cdot g_n\right)}{\left(2^{k-1}\cdot g_n\right)^{\alpha}}
 \leq n\cdot\frac{L\left(g_n\right)}{2^{\left(k-1\right)\cdot \alpha/2}\cdot g_n^{\alpha}}
 =\frac{n\cdot \left(1-F\left(g_n\right)\right)}{2^{\left(k-1\right)\cdot \alpha/2}},\label{eq: n mu chi 2ka}
\end{align}
for $n$ sufficiently large uniformly in $k\geq 2$.
Next, we aim to prove that $(1-F\left(g_n\right))\sim b_{n}/n$.
On the one hand we have that $F(F^{\leftarrow}(x))\geq x$ 
and on the other hand $F(F^{\leftarrow}(x)-1)\leq x$.
Hence, the definition of $(g_n)$ in \eqref{eq: un} gives
\begin{align}
 \frac{b_n}{n}&\sim \frac{b_n-\zeta_n}{n}
 \geq 1-F\left(F^{\leftarrow}\left(1-\frac{b_{n}-\zeta_n}{n}\right)\right)
 = 1-F\left(g_n\right)
 = \frac{L\left(g_n\right)}{\left(g_n\right)^{\alpha}}
  \sim \frac{L\left(g_n-1\right)}{\left(g_n-1\right)^{\alpha}}\notag\\
 &=1-F\left(g_n-1\right)
 = 1-F\left(F^{\leftarrow}\left(1-\frac{b_{n}-\zeta_n}{n}\right)-1\right)
 \geq \frac{b_n-\zeta_n}{n}
 \sim \frac{b_n}{n}\label{eq: bn asymp}
\end{align}
and thus
\begin{align}
 n\cdot\mu\left(\chi>2^{k-1}\cdot g_n\right)<\frac{b_n\cdot 2^{3/4\cdot \alpha}}{2^{k\cdot \alpha/2}},\label{eq: n mu chi 2k}
\end{align}
for $n$ sufficiently large uniformly in $k\geq 2$
and thus
\begin{align}
 b_n-2r+1-n\cdot\mu\left(\chi>2^{k-1}\cdot g_n\right)
 \geq \frac{2^{\alpha/4}-1}{2^{\alpha/4}}\cdot b_n-2r+1
 \geq \frac{2^{\alpha/4}-1}{2^{\alpha}}\cdot b_n,\label{eq: bn-2r+1 ineq}
\end{align}
for $n$ sufficiently large uniformly in $k\geq 2$.

In order to proceed we will make use of the following lemma.
\begin{lemma}[{\cite[Lemma 2.9]{kessebohmer_strong_2018}}]\label{lemma: Tnfn chi deviation allg}
Let $\left(\Omega, \mathcal{A},  T, \mu,\mathcal{F},\left\|\cdot\right\|\right)$ fulfill Property $\mathfrak{C}$.
Then there exist positive constants $K$, $N$, $U$
such that for all 
$\varphi\in\mathcal{F}$ fulfilling $\int\varphi\;\mathrm{d}\mu=0$, all 
$u\in\mathbb{R}_{>0}$, 
and all $n\in\N_{>N}$ we have
\begin{align*}
\mu\left(\max_{i\leq n}\left|\mathsf{S}_i\varphi\right|\geq  u\right)
&\leq K\cdot \exp\left(-U\cdot \frac{ u}{\left\|\varphi\right\|}\cdot\min\left\{\frac{u}{n\cdot\left|\varphi\right|_1},1\right\}\right).
\end{align*}
\end{lemma}

Combining  \eqref{eq: Phi kn 1st estimate}, \eqref{eq: bn-2r+1 ineq}, and
Lemma \ref{lemma: Tnfn chi deviation allg} yields
\begin{align}
\mu\left(\Phi_{k,n}'\right)
&\leq\mu\left(\left|\sum_{i=1}^n\overline{\mathbbm{1}}_{\left\{\chi\circ T^{i-1}>2^{k-1}\cdot g_n\right\}}\right|> \frac{2^{\alpha/4}-1}{2^{\alpha}}\cdot b_n\right)\notag\\
&\leq K\cdot \exp\left(-U\cdot \frac{2^{\alpha/4}-1}{2^{\alpha}}\cdot \frac{b_n}{\left\|\overline{\mathbbm{1}}_{\left\{\chi>2^{k-1}\cdot g_n\right\}}\right\|}\right.\notag\\
&\qquad\qquad\left.\cdot\min\left\{\frac{2^{\alpha/4}-1}{2^{\alpha}}\cdot \frac{b_n}{n\cdot\left|\overline{\mathbbm{1}}_{\left\{\chi>2^{k-1}\cdot g_n\right\}}\right|_1},1\right\}\right),\label{eq: estim Phikn'}
\end{align}
for $n$ sufficiently large uniformly in $k\geq 2$.
Furthermore,  we have that 
\begin{align*}
\left|\overline{\mathbbm{1}}_{\left\{\chi>2^{k-1}\cdot g_n\right\}}\right|_1
&= 2\cdot\frac{L\left(2^{k-1}\cdot g_n\right)}{\left(2^{k-1}\cdot g_n\right)^{\alpha}}\cdot\left(1-\frac{L\left(2^{k-1}\cdot g_n\right)}{\left(2^{k-1}\cdot g_n\right)^{\alpha}}\right)
\leq 2\cdot\frac{L\left(2^{k-1}\cdot g_n\right)}{\left(2^{k-1}\cdot g_n\right)^{\alpha}}
\end{align*}
and \eqref{eq: n mu chi 2k} implies 
\begin{align}
\left|\overline{\mathbbm{1}}_{\left\{\chi>2^{k-1}\cdot g_n\right\}}\right|_1
&\leq 4\cdot\frac{b_n}{2^{\alpha/2\cdot k}\cdot n},\label{eq: 1 norm 1}
\end{align}
for $n$ sufficiently large uniformly in $k\geq 2$.
With $K_{1}$ defined in \eqref{C 1} we get
\begin{align}
 \left\|\overline{\mathbbm{1}}_{\left\{\chi>2^{k-1}\cdot g_n\right\}}\right\|
 &\leq \left\|\mathbbm{1}_{\left\{\chi>2^{k-1}\cdot g_n\right\}}\right\|+\left\|\mu\left(\chi>2^{k-1}\cdot g_n\right)\right\|\notag\\
 &=\left\|\mathbbm{1}_{\left\{\chi>2^{k-1}\cdot g_n\right\}}\right\|+\mu\left(\chi>2^{k-1}\cdot g_n\right)\left\|\mathbbm{1}\right\|
 \leq K_{1}+\left\|\mathbbm{1}\right\|\label{eq: norm overline 1}
\end{align}
and thus an application of \eqref{eq: estim Phikn'}, \eqref{eq: 1 norm 1}, and \eqref{eq: norm overline 1} gives
\begin{align}
\mu\left(\Phi_{k,n}'\right)
&\leq K\cdot \exp\left(-U\cdot \frac{2^{\alpha/2}-1}{2^{\alpha}}\cdot \frac{b_n}{K_{1}+\left\|1\right\|}\cdot\min\left\{\frac{\left(2^{\alpha/2}-1\right)\cdot 2^{\alpha/2\cdot k}}{4},1\right\}\right)\notag\\
&= K\cdot\exp\left(-W_{\alpha}\cdot b_n\right),\label{eq: mu Phi kn estim}
\end{align}
for $n$ sufficiently large uniformly in $k\geq 2$, where 
\begin{align*}
 W_{\alpha}\coloneqq U\cdot \frac{\left(2^{\alpha/2}-1\right)^2}{2^{\alpha+2}\cdot\left(K_{1}+\left\|1\right\|\right)}.
\end{align*}
By using this estimate and \eqref{eq: Tn 2k Potter} we can further estimate
\begin{align}
 \int \mathsf{T}_{n}^{2^{k}\cdot g_n}\chi\;\mathrm{d}\mu\cdot\mu\left(\Phi_{k,n}'\right)
 &\leq K\cdot2^{k+1}\cdot \exp\left(-W_{\alpha}\cdot b_n\right)\cdot \int \mathsf{T}_{n}^{g_n}\chi\;\mathrm{d}\mu.\label{eq: int T mu Phi 1}
\end{align}
This will later give us the estimate for small $k$, let us proceed with an estimate for large values of $k$.
Let $m\coloneqq \left\lfloor 2/\alpha\right\rfloor +1$.
There exists $N\in\mathbb{N}$ such that for all $n\geq N$ and $k\in\mathbb{N}$ we have 
\begin{align}
 \Phi_{k,n}'\subset \left\{\# \left\{j\leq n\colon \chi\circ T^{j-1} >2^{k-1}\cdot g_n\right\}\geq m  \cdot r\right\}
 \eqqcolon  \Upsilon_{k,n}.\label{eq: def Upsilon knm}
\end{align}
We have that 
\begin{align*}
 \Upsilon_{k,n}\subset \bigcup_{\substack{ 1\leq i_1<\ldots<i_m\leq n\\ i_j-i_{j-1}\geq r, j=2,\ldots,n}}\bigcap_{j=1}^m\left\{\chi\circ T^{i_j-1}>2^{k-1}\cdot g_n\right\}
\end{align*}
and thus
\begin{align*}
 \mu\left(\Upsilon_{k,n}\right)\leq \sum_{\substack{ 1\leq i_1<\ldots<i_m\leq n\\ i_j-i_{j-1}\geq r, j=2,\ldots,n}}\mu\left(\bigcap_{j=1}^m\left\{\chi\circ T^{i_j-1}>2^{k-1}\cdot g_n\right\}\right).
\end{align*}
Each of the summands can be estimated using \eqref{eq: int XY estim psi}:
\begin{align*}
 \MoveEqLeft\mu\left(\bigcap_{j=1}^m\left\{\chi\circ T^{i_j-1}>2^{k-1}\cdot g_n\right\}\right)\\
 &\leq \mu\left(\chi\circ T^{i_1-1}>2^{k-1}\cdot g_n\right)\cdot \mu\left(\bigcap_{j=2}^m\left\{\chi\circ T^{i_j-1}>2^{k-1}\cdot g_n\right\}\right)\cdot\left(1+\bm{\psi}\left(i_2-i_1\right)\right)\\
 &\leq \mu\left(\chi\circ T^{i_1-1}>2^{k-1}\cdot g_n\right)\cdot \mu\left(\bigcap_{j=2}^m\left\{\chi\circ T^{i_j-1}>2^{k-1}\cdot g_n\right\}\right)\cdot\left(1+\bm{\psi}\left( r\right)\right)\\
 &\;\;\vdots\\
 &\leq \prod_{j=1}^m\mu\left(\chi\circ T^{i_j-1}>2^{k-1}\cdot g_n\right)\cdot\left(1+\bm{\psi}\left( r\right)\right)^{m-1}\\
 &=\mu\left(\chi>2^{k-1}\cdot g_n\right)^m\cdot\left(1+\bm{\psi}\left( r\right)\right)^{m-1}\\
 &\leq \mu\left(\chi>2^{k-1}\cdot g_n\right)^m\cdot 2^{m-1}.
\end{align*}
We have that 
\begin{align*}
 \#\left\{\left(i_1,\ldots,i_m\right)\colon 1\leq i_1<\ldots<i_m\leq n\right\}
 =\binom{n}{m}<n^m
\end{align*}
implying
\begin{align*}
 \mu\left(\Upsilon_{k,n}\right)
 &\leq \left(n\cdot \mu\left(\chi>2^{k-1}\cdot g_n\right)\right)^m\cdot 2^{m-1}.
\end{align*}
With \eqref{eq: n mu chi 2k} we conclude
\begin{align*}
 \mu\left(\Upsilon_{k,n}\right)
 &\leq \left(\frac{2\cdot b_n}{2^{k\cdot \alpha/2}}\right)^m\cdot 2^{m-1},
\end{align*}
 for $n$ sufficiently large uniformly in $k\geq 2$.
Combining this estimate with \eqref{eq: Tn 2k Potter} and \eqref{eq: def Upsilon knm} gives
\begin{align}
 \int \mathsf{T}_{n}^{2^{k}\cdot g_n}\chi\;\mathrm{d}\mu\cdot\mu\left(\Phi_{k,n}'\right)
 &\leq  2^{2m}\cdot 2^{k\cdot\left(1-\alpha/2\cdot m\right)}\cdot b_n^m \cdot \int \mathsf{T}_{n}^{g_n}\chi\;\mathrm{d}\mu,\label{eq: int T mu Phi 2}
\end{align}
for $n$ sufficiently large uniformly in $k\geq 2$.

Finally, we combine  the estimates for small and large $k$ given in \eqref{eq: int T mu Phi 1} and \eqref{eq: int T mu Phi 2}  to obtain
\begin{align}
\MoveEqLeft\sum_{k=2}^{\infty}\int \mathsf{T}_{n}^{2^{k}\cdot g_n}\chi\;\mathrm{d}\mu\cdot\mu\left(\Phi_{k,n}'\right)\notag\\
&\leq \left(\sum_{k=2}^{b_n^{1/2}-1} K\cdot2^{k+1}\cdot \exp\left(-W_{\alpha}\cdot b_n\right)
+\sum_{k=b_n^{1/2}}^{\infty} 2^{2m}\cdot 2^{k\cdot\left(1-\alpha/2\cdot m\right)}\cdot b_n^m\right) \cdot \int \mathsf{T}_{n}^{g_n}\chi\;\mathrm{d}\mu.\label{eq: sum int T mu Phi 2}
\end{align}
Estimating the sums of the first factor of \eqref{eq: sum int T mu Phi 2} separately implies
\begin{align}
\sum_{k=2}^{b_n^{1/2}-1} K\cdot2^{k+1}\cdot \exp\left(-W_{\alpha}\cdot b_n\right)
&< K\cdot2^{b_n^{1/2}+1}\cdot \exp\left(-W_{\alpha}\cdot b_n\right)\notag\\
&=K\cdot \exp\left(\log2\cdot\left(b_n^{1/2}+1\right)-W_{\alpha}\cdot b_n\right)\notag\\
&\leq K\cdot \exp\left(-W_{\alpha}/2\cdot b_n\right),\label{eq: sum k=2 bn1/2}
\end{align}
for $n$ sufficiently large and 
\begin{align}
 \sum_{k=b_n^{1/2}}^{\infty} 2^{2m}\cdot 2^{k\cdot\left(1-\alpha/2\cdot m\right)}\cdot b_n^m
 &\leq  2^{2m}\cdot\frac{2^{b_n^{1/2}\cdot\left(1-\alpha/2\cdot m\right)}}{1-2^{1-\alpha/2\cdot m}}\cdot b_n^m,\label{eq: sum k=bn1/2}
\end{align}
for $n$ sufficiently large. Since we chose $m>2/\alpha$, \eqref{eq: sum k=bn1/2} tends to zero for $n$ tending to infinity.
Combining \eqref{eq: sum k=2 bn1/2} and \eqref{eq: sum k=bn1/2} with \eqref{eq: sum int T mu Phi 2} 
proves the statement of \ref{en: A2}.

\emph{Proof of \ref{en: A3}}:
In order to consider the case $k=1$ we notice that by the definition of $(g_n)$ in \eqref{eq: un} 
 and the fact that $F\left(F^{\leftarrow}\left(x\right)\right)\leq x$
we have that
\begin{align*}
\mu\left(\chi> g_n\right)
=1-F\left(g_n\right)
=1-F\left(F^{\leftarrow}\left(1-\frac{b_{n}-\zeta_n}{n}\right)\right)
\leq 
\frac{b_n-\zeta_n}{n}
\end{align*}
and thus
$b_n-n\cdot\mu\left(\chi> g_n\right)\geq\zeta_n$. 
If we combine this with 
\eqref{eq: Phi kn 1st estimate}, we can conclude
\begin{align}
 \mu\left(\Phi_{1,n}'\right)
&\leq \mu\left(\left|\sum_{i=1}^n\overline{\mathbbm{1}}_{\left\{\chi\circ T^{i-1}>g_n\right\}}\right|\geq \zeta_n-2r +1\right)
\leq \mu\left(\left|\sum_{i=1}^n\overline{\mathbbm{1}}_{\left\{\chi\circ T^{i-1}>g_n\right\}}\right|\geq \zeta_n/2\right),\label{eq: mu Phi n1}
\end{align}
for $n$ sufficiently large.
Using Lemma \ref{lemma: Tnfn chi deviation allg} implies
\begin{align}
  \mu\left(\Phi_{1,n}'\right)
 &\leq K\cdot \exp\left(- \frac{ U\cdot\zeta_n}{2\cdot \left\|\overline{\mathbbm{1}}_{\left\{\chi>g_n\right\}}\right\|}\cdot\min\left\{\frac{\zeta_n}{2\cdot n\cdot\left|\overline{\mathbbm{1}}_{\left\{\chi>g_n\right\}}\right|_1},1\right\}\right),\label{eq: Phi 1nr ineq}
\end{align}
for $n$ sufficiently large.
Using \eqref{eq: bn asymp}, 
we have that 
\begin{align*}
\left|\overline{\mathbbm{1}}_{\left\{\chi>g_n\right\}}\right|_1
 \leq 2\mu\left(\chi>g_n\right)=2\cdot\left(1-F\left(g_n\right)\right)\leq 3 b_n/n, 
\end{align*} 
for $n$ sufficiently large.
Therefore,
\begin{align}
 \min\left\{\frac{\zeta_n}{2\cdot n\cdot\left|\overline{\mathbbm{1}}_{\left\{\chi>g_n\right\}}\right|_1},1\right\}
 &\geq \min\left\{\frac{\zeta_n}{n\cdot 6 b_n/n},1\right\}
 =\min\left\{\frac{\zeta_n}{6 b_n},1\right\}=\frac{\zeta_n}{6 b_n},\label{eq: min frac 1}
\end{align}
for $n$ sufficiently large.
Combining \eqref{eq: norm overline 1} with \eqref{eq: Phi 1nr ineq} and \eqref{eq: min frac 1} and using the definition of $(\zeta_n)$ yields for $n$ sufficiently
large that
\begin{align}
  \mu\left(\Phi_{1,n}'\right)
  &\leq K\cdot \exp\left(-\frac{ U\cdot\zeta_n}{2\left(K_{1}+\left\|\mathbbm{1}\right\|\right)}\cdot\frac{\zeta_n}{6 b_n}\right)
  =K\cdot \exp\left(-\frac{ U\cdot b_n^{1/3}}{12\left(K_{1}+\left\|\mathbbm{1}\right\|\right)}\right),\label{eq: mu Phi 1nr final ineq}
\end{align}
 for $n$ sufficiently large
which tends to zero for $n$ tending to infinity.
If we combine this with \eqref{eq: Tn 2k Potter}, we obtain \ref{en: A3}.\medskip

\emph{Proof of \ref{en: B}}:
We set $\overline{\Phi}_n\coloneqq \left\{ \#\left\{ i\leq n\colon \chi\circ T^{i-1}>g_n\right\} < n\cdot \mu\left(\chi>g_n\right)-\zeta_n\right\}$ and have that 
\begin{align*}
\int \mathsf{S}_{n}^{b_{n}}\chi\;\mathrm{d}\mu
& \geq\int \mathsf{S}_{n}^{b_{n}}\chi\cdot\mathbbm{1}_{\overline{\Phi}_n^c }\;\mathrm{d}\mu
\geq \int \mathsf{T}_{n}^{g_{n}}\chi\cdot\mathbbm{1}_{\overline{\Phi}_n^c }\;\mathrm{d}\mu- \left(b_n-n\cdot \mu\left(\chi>g_n\right)+\zeta_n\right)\cdot g_n\notag\\
&=\int \mathsf{T}_{n}^{g_{n}}\chi\;\mathrm{d}\mu-\int \mathsf{T}_{n}^{g_{n}}\chi\cdot\mathbbm{1}_{\overline{\Phi}_n}\;\mathrm{d}\mu- \left(b_n-n\cdot \mu\left(\chi>g_n\right)+\zeta_n\right)\cdot g_n.
\end{align*}
It is sufficient to show
\begin{align}
 \int \mathsf{T}_{n}^{g_{n}}\chi\cdot\mathbbm{1}_{\overline{\Phi}_n}\;\mathrm{d}\mu=o\left(\int \mathsf{T}_{n}^{g_{n}}\chi\;\mathrm{d}\mu\right)\label{eq: int Tngn o}
\end{align}
and 
\begin{align}
  \left(b_n-n\cdot \mu\left(\chi>g_n\right)+\zeta_n\right)\cdot g_n=o\left(\int \mathsf{T}_{n}^{g_{n}}\chi\;\mathrm{d}\mu\right).\label{eq: zetan gn o}
\end{align}
We start with showing \eqref{eq: zetan gn o}.
We have by \eqref{E S*} and the definition
of $F$ that  
\begin{align*}
\int\mathsf{T}_{n}^{g_{n}}\chi\;\mathrm{d}\mu & \sim\frac{\alpha}{1-\alpha}\cdot n\cdot g_{n}^{1-\alpha}\cdot L\left(g_{n}\right)
  =\frac{\alpha}{1-\alpha}\cdot n\cdot \left(1-F\left(g_{n}\right)\right)\cdot g_{n}.
\end{align*}
This together with \eqref{eq: bn asymp}
shows \eqref{eq: zetan gn o}.

Let us look at the asymptotic \eqref{eq: int Tngn o}.
Similarly as in the proof of \ref{en: A1}
set 
\begin{align*}
 \mathcal{K}_{n}^{i}&\coloneqq\left\{J=(j_m)_{1\leq m\leq n}\in\left\{1,2,3\right\}^n\colon j_m\neq 3\text{ for }m\in \Gamma_{n,i}\text{ and }j_m=3\text{ for }m\in (\Gamma_{ n,i})^c\right.\\
 &\qquad\left.\text{ and }\#\left\{m\in  \Gamma_{n,i}\colon \chi\circ T^{m-1}> g_n\right\}<b_n-2\zeta_n+2r-1\right\}.
\end{align*}
This implies 
\begin{align*}
  \overline{\Phi}_n\subset\biguplus_{J\in\mathcal{K}_{n}^i}D_{k,n,J}\subset \left\{\#\left\{m\leq n\colon \chi\circ T^{m-1}> g_n\right\}< n\cdot\mu\left(\chi>g_n\right)-\zeta_n+2r-1\right\}\eqqcolon \overline{\Phi}_{n}'.
\end{align*}
 
Using an analogous argument  as in \ref{en: A1} we obtain with the help of \eqref{eq: int chi 1D < psi int chi int 1D} 
that 
\begin{align}
 \int \mathsf{T}_{n}^{g_{n}}\chi\cdot\mathbbm{1}_{\overline{\Phi}_{n}}\;\mathrm{d}\mu
 &\leq\sum_{i=1}^n\sum_{J\in\mathcal{K}_n^i}\upperleft{g_{n}}{}{\chi}\circ T^{i-1}\cdot\mathbbm{1}_{D_{k,n,J}}\;\mathrm{d}\mu\notag\\
 &\leq\sum_{i=1}^n\sum_{J\in\mathcal{K}_n^i}\frac{\left(1+\bm{\psi}\left(r\right)\right)^2}{1-{\bm \psi}\left(r\right)}\cdot\int\upperleft{g_{n}}{}{\chi}\circ T^{i-1}\;\mathrm{d}\mu
  \cdot \mu\left(D_{k,n,J}\right)\notag\\
 &\leq\frac{\left(1+\bm{\psi}\left(r\right)\right)^2}{1-{\bm \psi}\left(r\right)}\cdot\int\mathsf{T}_{n}^{g_n}\chi\;\mathrm{d}\mu
  \cdot \mu\left(\overline{\Phi}_{n}'\right) .\label{eq: int Tn Gamman}
\end{align}

In the next steps we estimate $\mu\left(\overline{\Phi}_{n}'\right)$. 
 We have that 
\begin{align*}
 \mu\left(\overline{\Phi}_{n}'\right)
 & = \mu\left(\sum_{i=1}^n\overline{\mathbbm{1}}_{\left\{\chi\circ T^{i-1}>g_n\right\}}< -\zeta_n+2r-1\right)
 \leq \mu\left(\left|\sum_{i=1}^n\overline{\mathbbm{1}}_{\left\{\chi\circ T^{i-1}>g_n\right\}}\right|> \zeta_n/2\right),
\end{align*}
for $n$ sufficiently large.
Hence, \eqref{eq: mu Phi n1} implies that we can estimate $\mu(\overline{\Phi}_{n}')$ in the same manner as $\mu\left(\Phi_{1,n}'\right)$
and obtain by \eqref{eq: mu Phi 1nr final ineq} that we have for sufficiently large $n$ that $\mu(\overline{\Phi}_{n}')\leq \exp( U\cdot b_n^{1/3}/(12(K_{1}+\left\|\mathbbm{1}\right\|)))$ 
which tends to zero  for $n$ tending to infinity.
Combining this observation with \eqref{eq: int Tn Gamman} 
proves \eqref{eq: int Tngn o} which was the final step in the proof of \ref{en: B}.
\end{proof}

\begin{proof}[Proof of Lemma \ref{lem: conv prob}]
We make use of the following upper estimate:
\begin{align}
 \mu\left(\left|\frac{\mathsf{S}_n^{b_n}\chi}{d_n}-1\right|>\epsilon\right)
 &\leq \mu\left(\frac{\mathsf{S}_n^{b_n}\chi}{d_n}-1>\epsilon\right)
 + \mu\left(\frac{\mathsf{S}_n^{b_n}\chi}{d_n}-1<-\epsilon\right)\label{eq: Snbnchi estim}
\end{align}
and we will show that both terms on the right-hand side tend to zero for $n$ tending to infinity.
In order to estimate the first summand of \eqref{eq: Snbnchi estim} we note that 
\begin{align*}
 \mu\left(\frac{\mathsf{S}_n^{b_n}\chi}{d_n}-1>\epsilon\right)
 &\leq \mu\left(\mathsf{T}_n^{g_n}\chi<\mathsf{S}_n^{b_n}\chi\right)
  +\mu\left(\frac{\mathsf{T}_n^{g_n}\chi}{d_n}-1>\epsilon\right).
\end{align*}
With $\Phi_{k,n}$ given in \eqref{eq: def Phi kn} we also have $\left\{\mathsf{T}_n^{g_n}\chi<\mathsf{S}_n^{b_n}\chi\right\}=\Phi_{1,n}$.
An application of \eqref{eq: Un bn1} and \eqref{eq: mu Phi kn estim}
implies that $\mu\left(\Phi_{1,n}\right)$ tends to zero for $n$ tending to infinity. 

In order to estimate $\mu\left(\mathsf{T}_n^{g_n}\chi/d_n-1>\epsilon\right)$
we note that \eqref{eq: asymptotic in A} implies $d_n\sim \int\mathsf{T}_n^{g_n}\chi\;\mathrm{d}\mu$. 
If we set $\overline{\mathsf{T}}^r_n\chi=\mathsf{T}^r_n\chi-\int\mathsf{T}^r_n\chi\;\mathrm{d}\mu$,
then
\begin{align*}
 \mu\left(\frac{\mathsf{T}_n^{g_n}\chi}{d_n}-1>\epsilon\right)
 &\leq \mu\left(\overline{\mathsf{T}}_n^{g_n}\chi>\frac{\epsilon}{2}\cdot d_n\right),
\end{align*}
for $n$ sufficiently large,
and an application of Lemma \ref{lemma: Tnfn chi deviation allg} hence gives
\begin{align}
  \mu\left(\frac{\mathsf{T}_n^{g_n}\chi}{d_n}-1>\epsilon\right)
  &\leq K\cdot \exp\left(-U\cdot \frac{\epsilon\cdot d_n}{2\cdot \left\|\upperleft{g_n}{}{\overline{\chi}}\right\|}\cdot \min\left\{\frac{\epsilon\cdot d_n}{2\cdot n\cdot\left|\upperleft{g_n}{}{\overline{\chi}}\right|_1},1\right\}\right),\label{eq: Tngn estim}
\end{align}
 for $n$ sufficiently large
where $\upperleft{r}{}{\overline{\chi}}\coloneqq \upperleft{r}{}{\chi}-\int \upperleft{r}{}{\chi}\mathrm{d}\mu$. 
We note that the first part of \eqref{C 1} implies
\begin{align}
 \left\|\upperleft{g_n}{}{\overline{\chi}}\right\|
 &\leq \left\|\upperleft{g_n}{}{\chi}\right\|+\left\|\int\upperleft{g_n}{}{\chi}\mathrm{d}\mu\right\|
 \leq g_n\cdot K_{1}+ \int\upperleft{g_n}{}{\chi}\mathrm{d}\mu\cdot \left\|\mathbbm{1}\right\|
 \leq g_n\cdot \left(K_{1}+ \left\|\mathbbm{1}\right\|\right)\label{eq: norm estim}
\end{align}
and $n\cdot\left|\upperleft{g_n}{}{\overline{\chi}}\right|_1
\leq 2 n\cdot \int \upperleft{g_n}{}{\chi}\mathrm{d}\mu
=2 \int\mathsf{T}_n^{g_n}\chi\mathrm{d}\mu
\leq 3d_n$,
for $n$ sufficiently large, where the last inequality follows from \eqref{eq: asymptotic in A}.
Combining this with \eqref{eq: Tngn estim} and \eqref{eq: norm estim} yields 
\begin{align*}
   \mu\left(\frac{\mathsf{T}_n^{g_n}\chi}{d_n}-1>\epsilon\right)
   &\leq K\cdot \exp\left(-\frac{U\cdot\epsilon^2}{12\cdot\left(K_{1}+\left\|\mathbbm{1}\right\|\right)}
   \cdot  \frac{d_n}{g_n}\right),
\end{align*}
for $n$ sufficiently large  and $\epsilon$ sufficiently small. 
We obtain  with an analogous calculation as in \cite[Proof of Theorem 1.7]{kessebohmer_strong_2018}
that 
\begin{align*}
 g_n&=F^{\leftarrow}\left(1-\frac{b_n-\zeta_n}{n}\right)
 \sim \left(\frac{n}{b_n-\zeta_n}\right)^{1/\alpha}\cdot \left(L^{1/\alpha}\right)^{\#}\left(\left(\frac{n}{b_n-\zeta_n}\right)^{1/\alpha}\right)\\
 &\sim \left(\frac{n}{b_n}\right)^{1/\alpha}\cdot \left(L^{1/\alpha}\right)^{\#}\left(\left(\frac{n}{b_n}\right)^{1/\alpha}\right).
\end{align*}
Hence, by the definition of $(d_n)$ in \eqref{eq: def dn} we have that
$d_n/g_n\sim \alpha/\left(1-\alpha\right)\cdot b_n$
and 
\begin{align}
    \mu\left(\frac{\mathsf{T}_n^{g_n}\chi}{d_n}-1>\epsilon\right)
   &\leq K\cdot \exp\left(-\frac{U\cdot\epsilon^2\cdot \left(1-\alpha\right)}{13\cdot\left(K_{1}+\left\|\mathbbm{1}\right\|\right)\cdot \alpha}
   \cdot  b_n\right),\label{eq: mu Tngnchi/dn-1>eps}
\end{align}
for $n$ sufficiently large.
Consequently, the right-hand side tends to zero as  $(b_n)$ tends to infinity.

Next we will estimate the second summand of \eqref{eq: Snbnchi estim}.
For arbitrary $\delta>0$ we have that 
\begin{align*}
 \mu\left(\frac{\mathsf{S}_n^{b_n}\chi}{d_n}-1<-\epsilon\right)
 &\leq \mu\left(\mathsf{T}_{n}^{\left(1-\delta\right)\cdot g_n}\chi>\mathsf{S}_{n}^{b_{n}}\chi\right)
 + \mu\left(\frac{\mathsf{T}_n^{\left(1-\delta\right)\cdot g_n}\chi}{d_n}-1<-\epsilon\right).
\end{align*}
We first estimate 
\begin{align*}
 \mu\left(\mathsf{T}_{n}^{\left(1-\delta\right)\cdot g_n}\chi>\mathsf{S}_{n}^{b_{n}}\chi\right)
&=\mu\left(\sum_{i=1}^n\mathbbm{1}_{\left\{\chi\circ T^{i-1}>\left(1-\delta\right)\cdot g_n\right\}}< b_n\right)\notag\\
&=\mu\left(\sum_{i=1}^n\overline{\mathbbm{1}}_{\left\{\chi\circ T^{i-1}>\left(1-\delta\right)\cdot g_n\right\}}< b_n-n\cdot\mu\left(\chi>\left(1-\delta\right)\cdot g_n\right)\right).
\end{align*}
 By a calculation analogous to \eqref{eq: n mu chi 2ka} and \eqref{eq: bn asymp}
we obtain that 
$n\cdot\mu\left(\chi>\left(1-\delta\right)\cdot g_n\right)
 >b_n/\left(1-\delta\right)^ \alpha/2$,
for $n$ sufficiently large
from which we can conclude
\begin{align*}
  \mu\left(\mathsf{T}_{n}^{\left(1-\delta\right)\cdot g_n}\chi>\mathsf{S}_{n}^{b_{n}}\chi\right)
  &\leq \mu\left(\left|\sum_{i=1}^n\overline{\mathbbm{1}}_{\left\{\chi\circ T^{i-1}>\left(1-\delta\right)\cdot g_n\right\}}\right|> \left(\left(1-\delta\right)^{- \alpha/2}-1\right)\cdot b_n\right),
\end{align*}
for $n$ sufficiently large.
Applying Lemma \ref{lemma: Tnfn chi deviation allg} yields 
\begin{align*}
 \mu\left(\mathsf{T}_{n}^{\left(1-\delta\right)\cdot g_n}\chi>\mathsf{S}_{n}^{b_{n}}\chi\right)
  &\leq K\cdot\exp\left( -U\cdot \left(\left(1-\delta\right)^{-\alpha/2}-1\right)\cdot\frac{b_n}{\left\|\overline{\mathbbm{1}}_{\left\{\chi>\left(1-\delta\right)\cdot g_n\right\}}\right\|}\right.\\
  &\qquad\left.\cdot \min\left\{\left(\left(1-\delta\right)^{-\alpha/2}-1\right)\cdot \frac{b_n}{n\cdot \left|\overline{\mathbbm{1}}_{\left\{\chi>\left(1-\delta\right)\cdot g_n\right\}}\right|_1}  ,1\right\}\right),
\end{align*}
 for $n$ sufficiently large.

Similarly as in \eqref{eq: 1 norm 1} we obtain
$\left|\overline{\mathbbm{1}}_{\left\{\chi>\left(1-\delta\right)\cdot g_n\right\}}\right|_1
 \leq 4\cdot \left(1-\delta\right)^{-\alpha/2}\cdot b_n/n$,
for $n$ sufficiently large
and similarly as in \eqref{eq: norm overline 1} we obtain 
$\left\|\overline{\mathbbm{1}}_{\left\{\chi>\left(1-\delta\right)\cdot g_n\right\}}\right\|\leq K_{1}+\left\|\mathbbm{1}\right\|$. 
Setting 
\begin{align*}
 A_{\alpha,\delta}\coloneqq \frac{U\cdot \left(\left(1-\delta\right)^{-\alpha/2}-1\right)}{K_{1}+\left\|\mathbbm{1}\right\|}
 \cdot\min\left\{\frac{\left(1-\delta\right)^{\alpha/2}\cdot \left(\left(1-\delta\right)^{ -\alpha/2}-1\right)}{4},1\right\}
\end{align*}
implies
\begin{align*}
 \mu\left(\mathsf{T}_{n}^{\left(1-\delta\right)\cdot g_n}\chi>\mathsf{S}_{n}^{b_{n}}\chi\right)
 &\leq K\cdot \exp\left(-A_{\alpha,\delta}\cdot b_n\right),
\end{align*}
 for $n$ sufficiently large which tends to zero for $n$ tending to infinity.

Finally, we estimate 
$\mu\left(\mathsf{T}_n^{\left(1-\delta\right)\cdot g_n}\chi/d_n-1<-\epsilon\right)$.
From \eqref{E S*} and \eqref{eq: asymptotic in A} we can conclude that 
\begin{align*}
 \int \mathsf{T}_n^{\left(1-\delta\right)\cdot g_n}\chi\;\mathrm{d}\mu\sim \left(1-\delta\right)^{\alpha}\cdot d_n.
\end{align*}
Hence,
\begin{align*}
 \mu\left(\frac{\mathsf{T}_n^{\left(1-\delta\right)\cdot g_n}\chi}{d_n}-1<-\epsilon\right)
 &\leq \mu\left(\overline{\mathsf{T}}_n^{\left(1-\delta\right)\cdot g_n}\chi<-\left(\epsilon+\left(1-\delta/2\right)^{\alpha}-1\right)\cdot d_n\right),
\end{align*}
 for $n$ sufficiently large.
We can choose $\delta>0$ sufficiently small such that 
$\epsilon+\left(1-\delta/2\right)^{\alpha}-1\geq \epsilon/2$
and an application of Lemma \ref{lemma: Tnfn chi deviation allg} yields
\begin{align*}
 \mu\left(\frac{\mathsf{T}_n^{\left(1-\delta\right)\cdot g_n}\chi}{d_n}-1<-\epsilon\right)
  &\leq K\cdot \exp\left(-U\cdot \frac{\epsilon\cdot d_n}{2\cdot \left\|\upperleft{\left(1-\delta\right)\cdot g_n}{}{\overline{\chi}}\right\|}\cdot \min\left\{\frac{\epsilon\cdot d_n}{2\cdot n\cdot\left|\upperleft{\left(1-\delta\right)\cdot g_n}{}{\overline{\chi}}\right|_1},1\right\}\right).
\end{align*}
Similarly as above we have that 
$\left\|\upperleft{\left(1-\delta\right)\cdot g_n}{}{\overline{\chi}}\right\|\leq \left(1-\delta\right)\cdot g_n\cdot\left(K_{1}+\left\|\mathbbm{1}\right\|\right)<g_n\cdot\left(K_{1}+\left\|\mathbbm{1}\right\|\right)$
and $\left|\upperleft{\left(1-\delta\right)\cdot g_n}{}{\overline{\chi}}\right|_1\leq \left|\upperleft{g_n}{}{\overline{\chi}}\right|_1\leq 3d_n$
and with an analogous estimation as the one leading to \eqref{eq: mu Tngnchi/dn-1>eps} we obtain 
\begin{align*}
  \mu\left(\frac{\mathsf{T}_n^{\left(1-\delta\right)\cdot g_n}\chi}{d_n}-1<-\epsilon\right)
  &\leq K\cdot\exp\left(-\frac{U\cdot \epsilon^2\cdot \left(1-\alpha\right)}{13\left(K_{1}+\left\|\mathbbm{1}\right\|\right)\cdot \alpha}\cdot b_n\right),
\end{align*}
for $n$ sufficiently large  and $\epsilon$ sufficiently small, which tends to zero for $n$ tending to infinity. 
\end{proof}

\subsection{Proof of Theorem \ref{thm: counterex}}\label{subsec: proof thm counterex}
\begin{proof}[Proof of Theorem \ref{thm: counterex}]
 Since the Lebesgue measure is $T$-invariant, we notice that the system $([0,1),\mathcal{B},T,\mu, BV,\left\|\cdot\right\|_{BV},\chi)$ is 
 precisely the system of piecewise expanding interval maps 
 covered in \cite[Section 1.4]{kessebohmer_strong_2018}
  with $BV$ defined as in Definition \ref{def: bound var}.
 The last condition to check is that there exists a constant $K_{1}$   such that for all $\ell>0$ we have that
 $\left\|\upperleft{\ell}{}{\chi}\right\|_{BV}\leq K_{1}\cdot \ell$ and 
 $\left\|\mathbbm{1}_{\{\chi>\ell\}}\right\|_{BV}\leq K_{1}$
 which is obviously fulfilled for our choice of $\chi$.

 Let $\omega\in \left[0, x\right)$ with $x\leq 2^{-b_n}$,
 then 
 \begin{align*}
  \chi\left(\omega\right)&=\omega^{-\gamma},
  \left(\chi\circ T\right)\left(\omega\right)=\left(\omega\cdot 2^{-1}\right)^{-\gamma},
  \ldots, 
  \left(\chi\circ T^{b_n}\right)\left(\omega\right)=\left(\omega\cdot 2^{-b_n}\right)^{-\gamma}
 \end{align*}
 yielding $\mathsf{S}_n^{b_n}\chi\left(\omega\right)\leq \left(\chi\circ T^{b_n}\right)\left(\omega\right)=\left(\omega\cdot 2^{-b_n}\right)^{-\gamma}$. 
 Furthermore, $\left(\chi\circ T^{b_n}\right)\left(\omega\right)$ is decreasing in $\omega$. 
 Hence, $\mu\left(\mathsf{S}_n^{b_n}\chi\geq \left(\omega\cdot 2^{-b_n}\right)^{-\gamma}\right)\geq \omega$
 and thus 
\begin{align*}
  \int\mathsf{S}_n^{b_n}\chi\;\mathrm{d}\mu
  &\geq \mu\left(\mathsf{S}_n^{b_n}\chi\geq \left(\omega\cdot 2^{-b_n}\right)^{-\gamma}\right)\cdot \left(\omega\cdot 2^{-b_n}\right)^{-\gamma}
  \geq \omega^{1-\gamma}\cdot 2^{-b_n}.
\end{align*}
 Since $\omega$ can be chosen arbitrarily small, this implies $\int\mathsf{S}_n^{b_n}\chi\;\mathrm{d}\mu=\infty$.
\end{proof}

\subsection{Proof of the statement in Remark \ref{rem: iid special case}}\label{subsec: proof remark}
Last, we show how our results carry over to the i.i.d. case. 
\begin{proof}[Proof of the statement in Remark \ref{rem: iid special case}]
Let $(X_n)$ be a sequence of i.i.d.\ random variables mapping $\Omega\to\mathbb{R}$ with probability measure $\mathbb{P}$.
Then we define $Y\colon \Omega\to\mathbb{R}^{\N}$ by
$Y(\omega)\coloneqq \left(X_1(\omega),X_2(\omega),\ldots\right)$.
Further, let 
$\mu\coloneqq Y_*\mathbb{P}=\mathbb{P}\circ Y^{-1}$.
Since the random variables $(X_n)$ are independent and identical distributed, $\mu$ can be written as 
$\mu=\mathbb{P}\circ X_1^{-1}\times \mathbb{P}\circ X_2^{-1}\times\ldots$.
The measure $\mu$ is invariant and mixing with respect to the dynamics 
obtained by the shift map $\sigma\colon \mathbb{R}^{\N}\to\mathbb{R}^{\N}$
given by $\sigma\left(x_1,x_2,\ldots\right)=\left(x_2,x_3,\ldots\right)$. 
If we write $x=x_1x_2\ldots$ and set $\chi\left(x\right)=x_1$, we obtain  
$\left(\chi\circ T^{n-1}\right) Y(\omega)= X_n(\omega)$.

Furthermore, we might introduce the Banach space of functions $\mathcal{F}$ on the shift space $\mathbb{R}^{\N}$
as all functions $f\in\mathcal{L}^{\infty}$ with $\left|\cdot\right|_{\infty}$ as a norm such that $f(x)$ is already determined by $x_1$.
Obviously, $\mathcal{F}$ is a Banach space which contains the constant functions 
and fulfills \eqref{ineq <l} and \eqref{C 0 f}.

Furthermore, the transfer operator $\widehat{T}$ of the transformation $T=\sigma$ has a spectral gap on $\mathcal{F}$.
This can be easily seen by considering that for all $f\in\mathcal{L}^1$ and $g\in \mathcal{L}^{\infty}$ we have \eqref{hat CYRI}.
In case that $f\in\mathcal{F}$ we even have that
$\int\widehat{T}f\cdot g\;\mathrm{d}\mu=\int f\;\mathrm{d}\mu\cdot \int g\circ T\;\mathrm{d}\mu$,
which follows from the fact that $f$ and $g\circ T$ are independent with respect to $\mu$.
If $\widehat{T}f=\int f\;\mathrm{d}\mu$, the above equality is fulfilled for all $g\in\mathcal{L}^{\infty}$
and since the transfer operator is uniquely defined, the equality $\widehat{T}f=\int f\;\mathrm{d}\mu$ has to hold.
Since $\int f\;\mathrm{d}\mu$ is a projection, we can write $\widehat{T}f=Pf$ and do not even need an additional operator $N$, 
i.e.\ we have an even stronger statement than a spectral gap. 
It is also immediately clear that in the i.i.d.\ case we have $\bm{\psi}\left(n\right)=0$, for all $n\in\mathbb{N}$.

If we set $\chi\left(x\right)=x_1$ as the observable,
then \eqref{C 1}   are fulfilled and we can apply all theorems to this system.
\end{proof}

\end{document}